\documentclass[preprint,3p,11pt,authoryear]{elsarticle}

\makeatletter
\def\ps@pprintTitle{
 \let\@oddhead\@empty
 \let\@evenhead\@empty
 \def\@oddfoot{\centerline{\thepage}}
 \let\@evenfoot\@oddfoot}
\makeatother

\usepackage{lmodern,amsthm,amsmath,natbib,amsfonts,amssymb,mathtools,mathrsfs,dsfont,xcolor,appendix}
\usepackage[colorlinks]{hyperref}

\usepackage{geometry}
\newtheorem{theorem}{Theorem}[section]
\newtheorem{lemma}{Lemma}[section]
\newtheorem{corollary}{Corollary}[section]
\theoremstyle{remark}

\numberwithin{equation}{section}

\allowdisplaybreaks[4]
\DeclareMathOperator{\td}{d}
\DeclareMathOperator{\te}{e}

\begin{document}

\begin{frontmatter}

\title{Logarithmic complete monotonicity of a matrix-parametrized analogue of the multinomial distribution}

\author[a1,a2]{Fr\'ed\'eric Ouimet\corref{cor1}}
\ead{frederic.ouimet2@mcgill.ca}
\author[a3,a4]{Feng Qi}
\ead{qifeng618@gmail.com}

\address[a1]{Department of Mathematics and Statistics, McGill University, Montreal, QC H3A 0B9, Canada}
\address[a2]{Division of Physics, Mathematics and Astronomy, California Institute of Technology, Pasadena, CA 91125, USA}
\address[a3]{Institute of Mathematics, Henan Polytechnic University, Jiaozuo 454003, Henan, China}
\address[a4]{School of Mathematical Sciences, Tianjin Polytechnic University, Tianjin 300387, China}

\cortext[cor1]{Corresponding author}

\begin{abstract}
In the paper, the authors introduce a matrix-parametrized generalization of the multinomial probability mass function that involves a ratio of several multivariate gamma functions.
They show the logarithmic complete monotonicity of this generalization and derive several inequalities involving ratios of multivariate gamma functions.
\end{abstract}


\end{frontmatter}

\section{Preliminaries}

Recall from~\cite[Chapter~XIII]{MR1220224}, \cite[Chapter~1]{MR2978140}, and~\cite[Chapter~IV]{MR0005923}, that an infinitely differentiable function $f$ is said to be completely monotonic on an interval $I$ if it has derivatives of all orders on $I$ and satisfies $(-1)^{n}f^{(n)}(x)\ge 0$ for all $x\in I$ and $n\in\mathbb{N}_0=\{0,1,2,\dotsc\}$.
Recall from~\cite[Definition~1]{MR2075188} and~\cite[Definition~5.10]{MR2978140} that an infinitely differentiable and positive function $f$ is said to be logarithmically completely monotonic on an interval $I$ if
\begin{equation*}
(-1)^{n} \frac{\td^{n}}{\td x^{n}} \ln f(x)\ge 0
\end{equation*}
for all $n\in\mathbb{N}=\{1,2,\dotsc\}$ and $x\in I$.
The property of being logarithmically completely monotonic is stronger than being completely monotonic, see~\cite[Theorem~1.1]{MR2112748}, \cite[Theorem~1]{MR2075188}, and~\cite[p.~627, (1.4)]{MR3367457}. When $I=(0,\infty)$, Bernstein's theorem (see, e.g., Theorem~12b in~\citep[p.~161]{MR0005923}) states that a function $f$ is completely monotonic on $(0,\infty)$ if and only if
\begin{equation}\label{Theorem12b-Laplace}
f(x)=\int_0^\infty \te^{-xs}\td\sigma(s)
\end{equation}
and the integral converges for all $x\in(0,\infty)$, where $\sigma(s)$ is nondecreasing on $(0,\infty)$. The integral representation~\eqref{Theorem12b-Laplace} equivalently says that a function $f$ is completely monotonic on $(0,\infty)$ if and only if it is the Laplace transform of $\sigma(s)$ on $(0,\infty)$. This is one of many reasons why researchers have been investigating (logarithmically) completely monotonic functions.
\par
The literature on this topic is far too extensive to cite here, but one specific kind of (logarithmically) completely monotonic functions, functions involving ratios of gamma functions, has attracted a lot of attention lately.
Recent contributions for this type of completely monotonic functions include~\cite{MR3730425, Berg-Cetinkaya-Karp-2021, MR3460546,MR3825458, MR3908972, MR4139119, MR4201158, MR4196670, MR4094493, MR3856139}.
\par
Let $m, r\in \mathbb{N}$ and let $\boldsymbol{M}$ be a matrix of order $m$. We use the notation $\boldsymbol{M}\ge 0$ (or $\boldsymbol{M}>0$, respectively) to indicate that the matrix $\boldsymbol{M}$ is positive semidefinite (or positive definite, respectively).
According to~\cite[Definition~6.2.1]{Gupta-Nagar-1999}, the $r$-tuple of real symmetric matrices, $(\boldsymbol{M}_1,\boldsymbol{M}_2,\dotsc,\boldsymbol{M}_r)$, is said to have the (type I) $(m\times m)$-matrix-variate Dirichlet distribution with parameters $(a_1,a_2,\dotsc,a_{r+1})\in (0,\infty)^{r+1}$ if its density function is given by
\begin{equation*}
\frac{\Gamma_m\bigl(\sum_{i=1}^{r+1} a_i\bigr)}{\prod_{i=1}^{r+1} \Gamma_m(a_i)} \Biggl|\boldsymbol{I}_m-\sum_{i=1}^r \boldsymbol{M}_i\Biggr|^{a_{r+1}-(m+1)/2} \prod_{i=1}^r |\boldsymbol{M}_i|^{a_i-(m+1)/2},
\end{equation*}
when $\boldsymbol{M}_i\ge 0$ for all $1\le i\le r$ and $\boldsymbol{I}_m-\sum_{i=1}^r \boldsymbol{M}_i\ge~0$, and is equal to $0$ otherwise, where
\begin{equation}
\begin{aligned}\label{Gamma(m)-Gamma-Eq}
\Gamma_m(z)
&=\int_{\boldsymbol{S}\in \mathrm{Sym}_m(\mathbb{R}) : \boldsymbol{S} >0} |\boldsymbol{S}|^{z-(m+1)/2} \exp(-\mathrm{tr}(\boldsymbol{S})) \td \boldsymbol{S}\\
&=\pi^{m(m-1)/4} \prod_{j=1}^m \Gamma\biggl(z-\frac{j-1}{2}\biggr), \quad \Re(z)>\frac{m-1}{2}
\end{aligned}
\end{equation}
denotes the multivariate gamma function, see~\cite[Section~35.3]{MR2723248} and~\cite{MR3189307}, and
\begin{equation*}
\Gamma(z)=\lim_{\ell\to\infty}\frac{\ell! \ell^z}{\prod_{k=0}^\ell(z+k)}, \quad z\in\mathbb{C}\setminus\{0,-1,-2,\dotsc\}
\end{equation*}
is the classical gamma function.
\par
Given the ties between the Dirichlet and multinomial distributions, a function that naturally generalizes the multinomial probability mass function is the following matrix-parametrized analogue
\begin{equation*}
P_{n,\boldsymbol{k}_r}(\boldsymbol{M}_1,\dotsc,\boldsymbol{M}_r)
=\frac{\Gamma_m\bigl(n+\frac{m+1}2\bigr)}{\prod_{i=1}^{r+1} \Gamma_m\bigl(k_i+\frac{m+1}2\bigr)} \Biggl|\boldsymbol{I}_m-\sum_{i=1}^r \boldsymbol{M}_i\Biggr|^{n-\|\boldsymbol{k}_r\|_1} \prod_{i=1}^r |\boldsymbol{M}_i|^{k_i},
\end{equation*}
where $n\in \mathbb{N}$, $\boldsymbol{k}_r = (k_1,k_2,\dotsc,k_r)\in\mathbb{N}_0^r \cap (n\mathcal{S}_r)$, $k_{r+1}=n-\|\boldsymbol{k}_r\|_1$, the matrices $\boldsymbol{M}_i$, as well as $\boldsymbol{I}_m-\sum_{i=1}^r \boldsymbol{M}_i$, are assumed to be symmetric and positive definite, and
\begin{equation*}
\mathcal{S}_r=\Biggl\{\boldsymbol{x}_r=(x_1,x_2,\dotsc,x_r)\in [0,1]^r : \|\boldsymbol{x}_r\|_1=\sum_{\ell=1}^{r}x_\ell\le1\Biggr\}
\end{equation*}
denotes the $r$-dimensional unit simplex.
\par
In this paper, we will show that $x\mapsto P_{xn, x\boldsymbol{k}_r}(\boldsymbol{M}_1,\dotsc,\boldsymbol{M}_r)$ is a logarithmically completely monotonic function on $(0,\infty)$. Hereafter, in Section~\ref{sec:combinatorial.inequalities}, we will derive some inequalities involving ratios of multivariate gamma functions.

\section{A lemma}

To reach our goal, we need to prove the following technical lemma.

\begin{lemma}\label{lem:Alzer.2018.lemma.1.generalization}
Let $\mathcal{S}_r^\circ$ denote the interior of $\mathcal{S}_r$, $\boldsymbol{u}_r=(u_1,u_2,\dotsc,u_r)\in\mathcal{S}_r^\circ$, and $u_{r+1}= 1-\|\boldsymbol{u}_r\|_1>0$. Then, for $\beta\ge 0$ and $y>1$,
\begin{equation}\label{eq:lem:Alzer.2018.lemma.1.generalization}
\mathcal{J}_{\beta,\boldsymbol{u}_r}(y)=\frac{1}{y^{\beta}(y-1)}-\sum_{i=1}^{r+1} \frac{1}{y^{\beta/u_i}(y^{1/u_i}-1)}>0.
\end{equation}
\end{lemma}

\begin{proof}[First proof]
For $y>1$, $\beta\ge 0$, and $u\in (0,\infty)$, let
\begin{equation}\label{H(y-beta)(u)}
H_{y,\beta}(u)=\frac{1}{y^{\beta/u}(y^{1/u}-1)}.
\end{equation}
Then straightforward calculations yield
\begin{align*}
\frac{\td^2}{\td u^2} H_{y,\beta}(u)
&=\frac{\ln y}{u^4(y^{1/u}-1)^3y^{\beta/u}}\bigl\{\bigl[\beta^2\bigl(y^{1/u}-1\bigr)^2+2 \beta \bigl(y^{1/u}-1\bigr) y^{1/u}\\
&\quad+\bigl(y^{1/u}+1\bigr) y^{1/u}\bigr]\ln y -2 u \bigl(y^{1/u}-1\bigr) \bigl[\beta \bigl(y^{1/u}-1\bigr)+y^{1/u}\bigr]\bigr\}\\
&=\frac{\ln y}{u^3(y^{1/u}-1)^3y^{\beta/u}} \bigl\{[\beta^2(t-1)^2+2 \beta (t-1) t+(t+1) t]\ln t\\
&\quad-2(t-1)[\beta(t-1)+t]\bigr\}\\
&\triangleq \frac{\ln y}{u^3(y^{1/u}-1)^3y^{\beta/u}} h_{\beta}(t),
\end{align*}
where $t=y^{1/u}>1$. Direct computations show us that
\begin{align*}
h_{\beta}'(t)&=\frac{(t-1)[\beta^2 (t-1)-2 \beta t-3 t]}{t}+[2 \beta^2 (t-1)+\beta (4 t-2)+2 t+1]\ln t,\\
h_{\beta}''(t)&=\frac{(t-1)[\beta^2 (3 t+1)+2 \beta t-t]}{t^2}+2 (\beta+1)^2 \ln t,\\
h_{\beta}'''(t)&=\frac{2 \beta^2(t^2+t+1)+2 \beta t (2 t+1)+t (2 t-1)}{t^3} > 0,
\end{align*}
and
\begin{equation*}
\lim_{t\to1^+}h_{\beta}''(t)=\lim_{t\to1^+}h_{\beta}'(t)=\lim_{t\to1^+}h_{\beta}(t)=0.
\end{equation*}
Accordingly, for any fixed $y>1$ and $\beta\ge 0$, the second derivative $H_{y,\beta}''(u)$ is positive on $(0,\infty)$. Hence, the function $u\mapsto H_{y,\beta}(u)$ is strictly convex on $(0,\infty)$ and the function $u\mapsto H_{y,\beta}(u)+H_{y,\beta}(1-u)$ is strictly convex on $(0,1)$.
From the limits
\begin{equation*}
\lim_{u\to 0^+} H_{y,\beta}(u)=0 \quad\text{and}\quad \lim_{u\to 1^-} H_{y,\beta}(u)=\frac{1}{y^{\beta}(y-1)},
\end{equation*}
we conclude that
\begin{equation*}
H_{y,\beta}(u)+H_{y,\beta}(1-u)<\frac{1}{y^{\beta}(y-1)},
\end{equation*}
which is equivalent to
\begin{equation}\label{eq:case.r.1}
\frac{1}{y^{\beta}(y-1)}-\frac{1}{y^{\beta/\|\boldsymbol{u}_r\|_1}\bigl(y^{1/\|\boldsymbol{u}_r\|_1}-1\bigr)}- \frac{1}{y^{\beta/(1-\|\boldsymbol{u}_r\|_1)}\bigl[y^{1/(1-\|\boldsymbol{u}_r\|_1)}-1\bigr]}>0.
\end{equation}
\par
The general case will follow by induction.
Indeed, assume that, for some integer $r\ge  2$, the inequality
\begin{equation}\label{eq:induction}
\frac{1}{z^{\beta}(z-1)}-\sum_{i=1}^r \frac{1}{z^{\beta/v_i}(z^{1/v_i}-1)}>0
\end{equation}
is valid, where $\beta\ge 0$, $z>1$, $(v_1,v_2,\dotsc,v_{r-1})\in \mathcal{S}_{r-1}^\circ$, and $v_r=1-\sum_{i=1}^{r-1} v_i>0$.
On the other hand, we can rewrite the function $\mathcal{J}_{\beta,\boldsymbol{u}_r}(y)$ defined in~\eqref{eq:lem:Alzer.2018.lemma.1.generalization} as
\begin{multline*}
\mathcal{J}_{\beta,\boldsymbol{u}_r}(y)
=\biggl[\frac{1}{y^{\beta/\|\boldsymbol{u}_r\|_1}(y^{1/\|\boldsymbol{u}_r\|_1}-1)}
-\sum_{i=1}^r \frac{1}{y^{\beta/u_i}(y^{1/u_i}-1)}\biggr]\\
+\biggl[\frac{1}{y^{\beta}(y-1)}-\frac{1}{y^{\beta/\|\boldsymbol{u}_r\|_1}(y^{1/\|\boldsymbol{u}_r\|_1}-1)}
-\frac{1}{y^{\beta/(1-\|\boldsymbol{u}_r\|_1)}[y^{1/(1-\|\boldsymbol{u}_r\|_1)}-1]}\biggr].
\end{multline*}
By the inequality~\eqref{eq:case.r.1}, the quantity in the second bracket is positive. By the induction hypothesis~\eqref{eq:induction} with $z=y^{1/\|\boldsymbol{u}_r\|_1}$ and $v_i=\frac{u_i}{\|\boldsymbol{u}_r\|_1}$, the quantity in the first bracket is positive.
This ends the first proof of Lemma~\ref{lem:Alzer.2018.lemma.1.generalization}.
\end{proof}

\begin{proof}[Second proof]
A function $\varphi(x)$ is said to be super-additive on an interval $I$ if the inequality $\varphi(x+y)\ge \varphi(x)+\varphi(y)$ holds for all $x,y\in I$ with $x+y\in I$.
A function $\varphi:[0,\infty)\to\mathbb{R}$ is said to be star-shaped if $\varphi(\nu t)\le\nu\varphi(t)$ for $\nu\in[0,1]$ and $t\ge 0$. See~\cite[Chapter~16]{Marshall-Olkin-Arnold} and~\cite[Section~3.4]{Niculescu-Persson-Monograph-2018}.
Between convex functions, star-shaped functions, and super-additive functions, there are the following relations:
\begin{enumerate}
\item
if $\varphi$ is convex on $[0,\infty)$ with $\varphi(0)\le0$, then $\varphi$ is star-shaped;
\item
if $\varphi:[0,\infty)\to\mathbb{R}$ is star-shaped, then $\varphi$ is super-additive.
\end{enumerate}
See~\cite[pp.~650--651, Section~B.9]{Marshall-Olkin-Arnold}. Shortly speaking, if $\varphi$ is convex on $[0,\infty)$ with $\varphi(0)\le0$, then $\varphi$ is super-additive.
\par
For $y>1$, $\beta\ge 0$, and $u\in (-\infty,\infty)$, define
\begin{equation*}
\varphi_{y,\beta}(u)=\begin{dcases}
0, & u=0;\\
\frac{1}{y^{\beta/u}(y^{1/u}-1)}, & u\ne 0.
\end{dcases}
\end{equation*}
From the facts that the function $H_{y,\beta}(u)$ defined in~\eqref{H(y-beta)(u)} is strictly convex on $(0,\infty)$, that the limit
\begin{equation*}
\lim_{u\to0^+}H_{y,\beta}(u)=0
\end{equation*}
is valid, and that $\varphi_{y,\beta}(u)=H_{y,\beta}(u)$ for $u>0$, we conclude that the function $u\mapsto\varphi_{y,\beta}(u)$ is convex on $[0,\infty)$ with $\varphi_{y,\beta}(0)=0$. This means that the function $u\in[0,\infty)\mapsto\varphi_{y,\beta}(u)$ is star-shaped and super-additive. Accordingly, it follows that
\begin{equation*}
\sum_{i=1}^{r+1}\varphi_{y,\beta}(u_i)\le \varphi_{y,\beta}\Biggl(\sum_{i=1}^{r+1}u_i\Biggr)=\varphi_{y,\beta}(1),
\end{equation*}
that is,
\begin{equation*}
\sum_{i=1}^{r+1} \frac{1}{y^{\beta/u_i}(y^{1/u_i}-1)}
\le \frac{1}{y^{\beta}(y-1)}.
\end{equation*}
The inequality~\eqref{eq:lem:Alzer.2018.lemma.1.generalization} is thus proved. The second proof of Lemma~\ref{lem:Alzer.2018.lemma.1.generalization} is complete.
\end{proof}

\begin{proof}[Third proof (due to G\'erard Letac)]
Since $u_i>0$ for $1\le i\le r+1$ and $\sum_{i=1}^{r+1} u_i=1$, it is sufficient to show that
\begin{equation*}
\frac{1}{y^{\beta/u_i} (y^{1/u_i}-1)} < \frac{u_i}{y^{\beta} (y-1)}, \quad 1 \le i \le r+1.
\end{equation*}
By writing
\begin{equation*}
s_i=\frac{1}{u_i} \quad \text{and} \quad g(s_i)=\frac{s_i}{y^{s_i \beta} (y^{s_i}-1)},
\end{equation*}
it is sufficient to show that
\begin{equation*}
g(s) < g(1), \quad s>1.
\end{equation*}
But this readily follows from
\begin{equation*}
\frac{\td}{\td s} \ln g(s)=\biggl[\frac{1}{\ln(y^s)}-\frac{y^s}{y^s-1}-\beta\biggr]\ln y
=\biggl(\frac{1}{x}-\frac{1}{1-\te^{-x}}-\beta\biggr)\ln y
<0,
\end{equation*}
where we made the substitution $x=\ln(y^s)>0$ for $s,y>1$ and we used the inequality $\te^{-x}>1-x$ for $x\in(0,\infty)$.
The third proof of Lemma~\ref{lem:Alzer.2018.lemma.1.generalization} is complete.
\end{proof}

\section{Logarithmic complete monotonicity}

Now we are in a position to state and prove our main result.

\begin{theorem}\label{thm:completely.monotonic}
Let $m,n,r\in \mathbb{N}$, $\boldsymbol{\alpha}_r=(\alpha_1, \alpha_2, \dotsc, \alpha_r)\in n\mathcal{S}_r^\circ$, $\alpha_{r+1}=n- \|\boldsymbol{\alpha}_r\|_1>0$, and let $\boldsymbol{M}_i$ for $0\le i\le r$ be symmetric matrices of order $m$, satisfying $\boldsymbol{M}_i\ge 0$ and $\boldsymbol{M}_{r+1}=\boldsymbol{I}_m-\sum_{i=1}^r \boldsymbol{M}_i\ge 0$. Then the function
\begin{equation}\label{eq:thm:completely.monotonic}
\mathcal{Q}(x)=\frac{\Gamma_m\bigl(xn+\frac{m+1}2\bigr)}{\prod_{i=1}^{r+1} \Gamma_m\bigl(x\alpha_i+\frac{m+1}2\bigr)} \prod_{i=1}^{r+1} |\boldsymbol{M}_i|^{x\alpha_i}
\end{equation}
is logarithmically completely monotonic on $(0,\infty)$.
\end{theorem}

\begin{proof}
Without loss of generality, we can assume that $\boldsymbol{M}_i>0$ for all $0\le i\le r+1$.
By taking the logarithm on both sides of the equation~\eqref{eq:thm:completely.monotonic}, we have
\begin{equation*}
\ln \mathcal{Q}(x)
=\ln \Gamma_m\biggl(xn+\frac{m+1}{2}\biggr)-\sum_{i=1}^{r+1} \ln \Gamma_m\biggl(x\alpha_i+\frac{m+1}{2}\biggr) +x \sum_{i=1}^{r+1}\alpha_i\ln |\boldsymbol{M}_i|.
\end{equation*}
Denote
\begin{equation*}
\psi_m(z)=\frac{\td}{\td z} \ln \Gamma_m(z)=\sum_{j=1}^m \psi\biggl(z-\frac{j-1}{2}\biggr).
\end{equation*}
Then a direct differentiation gives
\begin{equation}\label{eq:prop:completely.monotonic.derivative.h}
\begin{aligned}
[\ln \mathcal{Q}(x)]'
&=n\psi_m\biggl(xn+\frac{m+1}{2}\biggr)-\sum_{i=1}^{r+1}\alpha_i\psi_m\biggl(x\alpha_i+\frac{m+1}{2}\biggr) +\sum_{i=1}^{r+1}\alpha_i\ln|\boldsymbol{M}_i|\\
&=\sum_{j=1}^m \Biggl\{n\psi\biggl(xn+\frac{m-j}{2}+1\biggr)- \sum_{i=1}^{r+1}\alpha_i\psi\biggl(x\alpha_i+\frac{m-j}{2}+1\biggr)\Biggr\} \\
&\quad+\sum_{i=1}^{r+1}\alpha_i\ln|\boldsymbol{M}_i|.
\end{aligned}
\end{equation}
Using the special case $\ell=1$ of the integral representation
\begin{equation*}
\psi^{(\ell)}(z)=(-1)^{\ell+1}\int_0^{\infty} \frac{t^\ell \te^{-zt}}{1-\te^{-t}} \td t, \quad \Re(z)>0, \quad\ell\in\mathbb{N},
\end{equation*}
listed in~\cite[p.~260, 6.4.1]{MR0167642}, we obtain
\begin{align*}
[\ln \mathcal{Q}(x)]''
&=\sum_{j=1}^m \Biggl\{n^2 \psi'\biggl(xn+\frac{m-j}{2}+1\biggr)-\sum_{i=1}^{r+1} \alpha_i^2 \psi'\biggl(x\alpha_i+\frac{m-j}{2}+1\biggr)\Biggr\}\\
&=\sum_{j=1}^m \Biggl\{n^2 \int_0^{\infty} \frac{t}{\te^t-1}\exp\biggl[-\biggl(xn+\frac{m-j}{2}\biggr) t\biggr]\td t\\
&\quad-\sum_{i=1}^{r+1} \alpha_i^2 \int_0^{\infty} \frac{t}{\te^t-1}\exp\biggl[-\biggl(x\alpha_i+\frac{m-j}{2}\biggr)t\biggr] \td t\Biggr\} \\
&=\sum_{j=1}^m \Biggl\{\int_0^{\infty} \frac{se^{-xs}}{\te^{s/n}-1}\exp\biggl(-\frac{m-j}{2n}s\biggr)\td s\\
&\quad-\sum_{i=1}^{r+1} \int_0^{\infty}\frac{se^{-xs}}{\te^{s/\alpha_i}-1} \exp\biggl(-\frac{m-j}{2\alpha_i}s\biggr)\td s\Biggr\} \\
&=\sum_{j=1}^m \int_0^{\infty} \mathcal{J}_{(m-j)/2,\boldsymbol{\alpha}_r/n}\bigl(\te^{s/n}\bigr)se^{-xs} \td s,
\end{align*}
where $\mathcal{J}_{\beta,\boldsymbol{u}_r}(y)$ is defined in~\eqref{eq:lem:Alzer.2018.lemma.1.generalization}.
Applying Lemma~\ref{lem:Alzer.2018.lemma.1.generalization} yields
\begin{equation}\label{eq:prop:completely.monotonic.almost.equation}
(-1)^{n+1} \frac{\td^{n+1}}{\td x^{n+1}} \ln \mathcal{Q}(x)=\sum_{j=1}^m \int_0^{\infty} \mathcal{J}_{(m-j)/2,\boldsymbol{\alpha}_r/n}\bigl(\te^{s/n}\bigr) s^n \te^{-xs} \td s>0
\end{equation}
for $n\in\mathbb{N}$ and $x>0$.
\par
To prove the logarithmic complete monotonicity of $\mathcal{Q}(x)$ on $(0,\infty)$, it remains to verify that the positivity in~\eqref{eq:prop:completely.monotonic.almost.equation} is also valid for $n=0$. From the positivity of~\eqref{eq:prop:completely.monotonic.almost.equation} for $n=1$, it follows that the function $x\mapsto [-\ln \mathcal{Q}(x)]'$ is decreasing on $(0,\infty)$. Therefore, it is sufficient to show that $\lim_{x\to\infty} [-\ln \mathcal{Q}(x)]'\ge 0$.
\par
For $k\in\mathbb{N}_0$ and $a\ge 0$, the first result in~\cite[Lemma~2.8]{Alice-y-x-curvature.tex} reads as
\begin{equation}\label{polyg-diff-lim-ID1}
\lim_{x\to\infty}\bigl(x^{k+1}\bigl[\psi^{(k)}(x+a)-\psi^{(k)}(x)\bigr]\bigr)
=(-1)^kk! a.
\end{equation}
Hence, the case $k=0$ and $a=\frac{1}{2}$ in~\eqref{polyg-diff-lim-ID1} implies the asymptotic formula
\begin{equation*}
\psi\biggl(z+\frac{1}{2}\biggr)\sim\psi(z)+\frac{1}{2z}, \quad z\to\infty.
\end{equation*}
Therefore, we have
\begin{align*}
\psi\biggl(xn+\frac{m-j}{2}+1\biggr)
&\sim\psi\biggl(xn+\frac{m-j+1}{2}\biggr)+\frac{1}{2n\bigl(x+\frac{m-j+1}{2n}\bigr)}\\
&\sim\psi\biggl(xn+\frac{m-j}{2}\biggr)+\frac{1}{2n\bigl(x+\frac{m-j}{2n}\bigr)}+\frac{1}{2n\bigl(x+\frac{m-j+1}{2n}\bigr)}
\end{align*}
as $x\to\infty$. By induction, we obtain
\begin{equation}\label{psi-asym-eq1}
\psi\biggl(xn+\frac{m-j}{2}+1\biggr)\sim\psi(xn)+\frac{1}{2n}\sum_{\ell=0}^{m-j+1}\frac{1}{x+\frac{\ell}{2n}}, \quad x\to\infty.
\end{equation}
Similarly, we derive
\begin{equation}\label{psi-asym-eq2}
\psi\biggl(x\alpha_i+\frac{m-j}{2}+1\biggr)
\sim\psi(x\alpha_i)+\frac{1}{2\alpha_i} \sum_{\ell=0}^{m-j+1}\frac{1}{x+\frac{\ell}{2\alpha_i}}, \quad x\to\infty.
\end{equation}
Substituting~\eqref{psi-asym-eq1} and~\eqref{psi-asym-eq2} into \eqref{eq:prop:completely.monotonic.derivative.h} and simplifying lead to
\begin{align*}
[-\ln \mathcal{Q}(x)]'
&\sim m\Biggl[\sum_{i=1}^{r+1}\alpha_i\psi(x\alpha_i) -n\psi(xn)\Biggr]-\sum_{i=1}^{r+1}\alpha_i\ln|\boldsymbol{M}_i|\\
&\quad+\frac{1}{2}\sum_{j=1}^m\sum_{\ell=0}^{m-j+1}\Biggl(\sum_{i=1}^{r+1}\frac{1}{x+\frac{\ell}{2\alpha_i}} -\frac{1}{x+\frac{\ell}{2n}}\Biggr), \quad x\to \infty.
\end{align*}
Using the asymptotic formula
\begin{equation*}
\psi(z) \sim\ln z-\frac{1}{2z}, \quad z\to \infty,
\end{equation*}
derived from~\cite[p.~259, 6.3.18]{MR0167642}, we conclude that
\begin{align*}
[-\ln \mathcal{Q}(x)]'
&\sim m\Biggl[\sum_{i=1}^{r+1}\alpha_i\ln(x\alpha_i) -n\ln(xn)\Biggr]-\sum_{i=1}^{r+1}\alpha_i\ln|\boldsymbol{M}_i|\\
&\quad-\frac{mr}{2x}+\frac{1}{2}\sum_{j=1}^m\sum_{\ell=0}^{m-j+1}\Biggl(\sum_{i=1}^{r+1}\frac{1}{x+\frac{\ell}{2\alpha_i}} -\frac{1}{x+\frac{\ell}{2n}}\Biggr), \quad x\to \infty.
\end{align*}
Consequently, we have
\begin{equation}\label{eq:near.proof}
\begin{aligned}
\lim_{x\to\infty} [-\ln \mathcal{Q}(x)]'
&=mn\Biggl[\sum_{i=1}^{r+1}\frac{\alpha_i}{n}\ln\frac{\alpha_i}{n} -\sum_{i=1}^{r+1}\frac{\alpha_i}{n}\ln\Bigl(|\boldsymbol{M}_i|^{1/m}\Bigr)\Biggr]\\ &=-mn\sum_{i=1}^{r+1}\frac{\alpha_i}{n}\ln\biggl(\frac{|\boldsymbol{M}_i|^{1/m}}{\alpha_i/n}\biggr)\\
&\ge -mn\ln\Biggl(\sum_{i=1}^{r+1}|\boldsymbol{M}_i|^{1/m}\Biggr)\\
&\ge -mn\ln\Biggl(\biggl|\sum_{i=1}^{r+1}\boldsymbol{M}_i\biggr|^{1/m}\Biggr)\\
&=-mn\ln\bigl(|\boldsymbol{I}_m|^{1/m}\bigr)\\
&=0,
\end{aligned}
\end{equation}
where we used Jensen's inequality for the concave function $\ln x$ on $(0,\infty)$, see~\cite[Chapter~I, p.~6]{MR1220224}, and Minkowski's determinant inequality for the symmetric positive semidefinite matrices $\boldsymbol{M}_i$ of order $m$, see~\cite{MR223383} or~\cite[Chapter~VIII, p.~214, Theorem~2]{MR1220224}.
This ends the proof of Theorem~\ref{thm:completely.monotonic}.
\end{proof}

A trivial modification at the end of the last proof yields the following corollary.

\begin{corollary}\label{cor:completely.monotonic}
Let $m,n,r\in \mathbb{N}$, $\boldsymbol{\alpha}_r=(\alpha_1, \alpha_2, \dotsc, \alpha_r)\in n\mathcal{S}_r^\circ$, $\alpha_{r+1}=n- \|\boldsymbol{\alpha}_r\|_1>0$, and let $(p_1,p_2,\dotsc,p_r)\in \mathcal{S}_r$ with $p_{r+1} = 1 - \sum_{i=1}^r p_i$.
Then the function
\begin{equation}\label{eq:cor:completely.monotonic}
\widetilde{\mathcal{Q}}(x)=\frac{\Gamma_m\bigl(xn+\frac{m+1}2\bigr)}{\prod_{i=1}^{r+1} \Gamma_m\bigl(x\alpha_i+\frac{m+1}2\bigr)} \prod_{i=1}^{r+1} p_i^{x\alpha_i}
\end{equation}
is logarithmically completely monotonic on $(0,\infty)$.
\end{corollary}

\begin{proof}
In the proof of Theorem~\ref{thm:completely.monotonic}, replace $|\boldsymbol{M}_i|$ everywhere by $p_i$. On the third line of \eqref{eq:near.proof}, simply use Jensen's inequality,
\begin{equation}
\sum_{i=1}^{r+1} p_i^{1/m} \leq \left(\sum_{i=1}^{r+1} p_i\right)^{1/m} = 1^{1/m} = 1,
\end{equation}
to obtain the desired conclusion.
\end{proof}

\section{Inequalities involving ratios of multivariate gamma functions}\label{sec:combinatorial.inequalities}

Considering the factor involving the gamma functions in~\eqref{eq:thm:completely.monotonic} and making use of the formula~\eqref{Gamma(m)-Gamma-Eq}, we denote and write
\begin{equation*}
R(x)
=\frac{\Gamma_m\bigl(xn+\frac{m+1}2\bigr)}{\prod_{i=1}^{r+1} \Gamma_m\bigl(x\alpha_i+\frac{m+1}2\bigr)}
=\frac{\prod_{j=1}^m \Gamma\bigl(xn+\frac{m-j}{2}+1\bigr)}{\pi^{r m(m-1)/4} \prod_{i=1}^{r+1} \prod_{j=1}^m \Gamma\bigl(x\alpha_i+ \frac{m-j}{2}+1\bigr)}.
\end{equation*}
We are now ready to derive several inequalities for $R(x)$ involving ratios of multivariate gamma functions $\Gamma_m$.

\begin{theorem}\label{multivariate-gamma-ineq-thm}
Let $\ell\in \mathbb{N}$ and let $(\lambda_1,\dotsc,\lambda_{\ell})\in (0,\infty)^{\ell}$ such that $\sum_{k=1}^\ell \lambda_k=1$.
\begin{enumerate}
\item
For $(x_1,\dotsc,x_{\ell})\in (0,\infty)^{\ell}$, we have
\begin{equation}\label{R(x)-Jensen-ineq}
R\Biggl(\sum_{k=1}^\ell \lambda_k x_k\Biggr) \le \prod_{k=1}^\ell[R(x_k)]^{\lambda_k},
\end{equation}
where the equality holds if and only if $x_1=x_2=\dotsm=x_k$.
\item
For $(x_1,\dotsc,x_{\ell})\in (0,\infty)^{\ell}$, we have
\begin{equation}\label{R(x)-Supper-Additive-ineq}
\prod_{k=1}^\ell R(x_k) < R\Biggl(\sum_{k=1}^\ell x_k\Biggr).
\end{equation}
\item
For $\ell=3$ and $(x_1,x_2,x_3)\in(0,\infty)^3$, if $x_1 \le x_3$, then
\begin{equation}\label{R(x)-3-factor-ineq}
R(x_1+x_2) R(x_3) \le R(x_1) R(x_2+x_3),
\end{equation}
where the equality holds if and only if $x_1=x_3$.
\end{enumerate}
\end{theorem}

\begin{proof}
The logarithmic complete monotonicity in Theorem~\ref{thm:completely.monotonic} implies that the function $\mathcal{Q}$ is logarithmically convex on $(0,\infty)$. Consequently, we obtain
\begin{multline*}
\mathcal{Q}\Biggl(\sum_{k=1}^\ell \lambda_k x_k\Biggr)=R\Biggl(\sum_{k=1}^\ell \lambda_k x_k\Biggr) \prod_{i=1}^{r+1} |\boldsymbol{M}_i|^{\alpha_i\sum_{k=1}^\ell \lambda_k x_k}\\
\le \prod_{k=1}^\ell[\mathcal{Q}(x_k)]^{\lambda_k}
=\prod_{k=1}^\ell[R(x_k)]^{\lambda_k} \prod_{k=1}^\ell\Biggl[\prod_{i=1}^{r+1} |\boldsymbol{M}_i|^{x_k\alpha_i}\Biggr]^{\lambda_k}
\end{multline*}
which can be simplified as~\eqref{R(x)-Jensen-ineq}. Due to the logarithmic convexity, it is trivial to see that the equality in~\eqref{R(x)-Jensen-ineq} holds if and only if $x_1=x_2=\dotsm=x_k$.
\par
Lemma~3 in~\cite{MR3730425} states that, if $h:[0,\infty)\to(0,1]$ is differentiable and the logarithmic derivative $\frac{h'(x)}{h(x)}=[\ln h(x)]'$ is strictly increasing on $(0,\infty)$, then the strict inequality $h(x)h(y)<h(x+y)$ is valid for $x,y\in(0,\infty)$.
By the way, we notice that this lemma is a special case $\varphi(x)=\ln h(x)$ of the result concluded in the first paragraph in the second proof of Lemma~\ref{lem:Alzer.2018.lemma.1.generalization} in this paper. From this, we can inductively derive
\begin{equation}\label{h(x)-supper=add-ineq}
\prod_{j=1}^{\ell}h(x_j)<h\Biggl(\sum_{j=1}^{\ell}x_j\Biggr).
\end{equation}
The logarithmic complete monotonicity in Theorem~\ref{thm:completely.monotonic} implies that $\mathcal{Q}(x)$ is decreasing on $(0,\infty)$ and that the logarithmic derivative $\frac{\mathcal{Q}'(x)}{\mathcal{Q}(x)}=[\ln\mathcal{Q}(x)]'$ is strictly increasing on $(0,\infty)$. Since
\begin{equation*}
\lim_{x\to0^+}\mathcal{Q}(x)=\frac{\Gamma_m\bigl(\frac{m+1}2\bigr)}{\prod_{i=1}^{r+1} \Gamma_m\bigl(\frac{m+1}2\bigr)}
=\frac{1}{\bigl[\Gamma_m\bigl(\frac{m+1}2\bigr)\bigr]^r}
\end{equation*}
and
\begin{gather*}
\Gamma_m\biggl(\frac{m+1}2\biggr)=\pi^{(m-1)m/4} \prod _{j=1}^m \Gamma\biggl(1+\frac{m-j}{2}\biggr)\\
=\pi^{(m-1)m/4} \prod _{k=0}^{m-1} \Gamma\biggl(1+\frac{k}{2}\biggr)
\ge\pi^{(m-1)m/4}
\ge1
\end{gather*}
for $m\in\mathbb{N}$, from the decreasing property of $\mathcal{Q}(x)$ on $(0,\infty)$, we deduce that
\begin{equation*}
0<\mathcal{Q}(0)\triangleq\lim_{x\to0^+}\mathcal{Q}(x)\le1.
\end{equation*}
Consequently, by applying the inequality~\eqref{h(x)-supper=add-ineq} to $\mathcal{Q}(x)$, we have
\begin{multline*}
\prod_{j=1}^{\ell}\mathcal{Q}(x_j)
=\prod_{j=1}^{\ell}\Biggl[R(x_j) \prod_{i=1}^{r+1} |\boldsymbol{M}_i|^{x_j\alpha_i}\Biggr]
=\prod_{j=1}^{\ell}R(x_j) \prod_{i=1}^{r+1} |\boldsymbol{M}_i|^{\alpha_i\sum_{j=1}^{\ell}x_j}\\
<\mathcal{Q}\Biggl(\sum_{j=1}^{\ell}x_j\Biggr)
=R\Biggl(\sum_{j=1}^{\ell}x_j\Biggr)\prod_{i=1}^{r+1} |\boldsymbol{M}_i|^{\alpha_i\sum_{j=1}^{\ell}x_j},
\end{multline*}
which can be reformulated as the inequality~\eqref{R(x)-Supper-Additive-ineq}.
\par
As done in the proof of Theorem~3.5 in~\cite{MR4201158}, the inequality~\eqref{R(x)-3-factor-ineq} and the equality case follow by using the logarithmic complete monotonicity of the function $\mathcal{Q}(x)$ in Theorem~\ref{thm:completely.monotonic} and by adapting the proof of Corollary 3 in~\cite{MR3730425}. The proof of Theorem~\ref{multivariate-gamma-ineq-thm} is complete.
\end{proof}

\section{Acknowledgements and declarations}

\subsection{Acknowledgements}
The authors thank G\'erard Letac (Institut de Math\'ematiques de Toulouse, Universit\'e Paul Sabatier, France; gerard.letac@math.univ-toulouse.fr) for providing the third proof of Lemma~\ref{lem:Alzer.2018.lemma.1.generalization}.

\subsection{Funding}
F. Ouimet was supported by postdoctoral fellowships from the NSERC (PDF) and the FRQNT (B3X supplement and B3XR).

\subsection{Data availability statement}
No data were used to support this study.

\subsection{Conflicts of interest}
The authors declare no conflict of interest.

\subsection{Author contributions}
%The authors contributed equally to this work. All authors read and approved the final manuscript.
Writing – original draft, Fr\'ed\'eric Ouimet and Feng Qi; Writing – review \& editing, Fr\'ed\'eric Ouimet and Feng Qi.

\bibliographystyle{authordate1}
\bibliography{Ouimet-Qi-2022-bib}

\end{document}